\documentclass[12pt]{amsart}
\usepackage{latexsym}
\usepackage{amssymb} 
\usepackage{mathrsfs}
\usepackage{amsmath}
\usepackage{latexsym}
\usepackage{delarray}
\usepackage{amssymb,amsmath,amsfonts,amsthm,mathrsfs}

\usepackage{hyperref}
\renewcommand\eqref[1]{(\ref{#1})} 

 \newtheorem{thm}{Theorem}[section]
 \newtheorem{cor}[thm]{Corollary}
 \newtheorem{lem}[thm]{Lemma}
 \newtheorem{prop}[thm]{Proposition}
 \theoremstyle{definition}
 
 \theoremstyle{remark}
 \newtheorem{rem}[thm]{Remark}
 
 \numberwithin{equation}{section}
\newcommand{\half}{\frac{1}{2}}

\newcommand{\ene}{\mathbb{N}}

\newcommand{\er}{\mathbb{R}}
\newcommand{\ce}{\mathbb{C}}

\newcommand{\zn}{\mathbb{Z}^n}
\newcommand{\tn}{\mathbb{T}^n}

\newcommand{\bi}{\begin{itemize}}
\newcommand{\ei}{\end{itemize}}
\newcommand{\be}{\begin{enumerate}}
\newcommand{\ee}{\end{enumerate}}
\newcommand{\beq}{\begin{equation}}
\newcommand{\eq}{\end{equation}}

\newcommand{\cdxi}{\ce^{d_{\xi}\times d_{\xi}}}

\newcommand{\subll}{\mathcal{L}_{sub}}

\newcommand{\efeeg}{{\mathcal{F}}_G}

\newcommand{\cL}{\mathcal{L}}

\def\jp#1{{\left\langle{#1}\right\rangle}}

\def\Op{{{\rm Op}}}

\DeclareMathOperator{\Tr}{Tr}

\def\d{\mathrm d}

\def\Gh{{\widehat{G}}}

\def\dxi{{d_\xi}}

\def\HS{{\mathtt{HS}}}
\def\Rn{{{\mathbb R}^n}}
\def\Tn{{{\mathbb T}^n}}
\def\Zn{{{\mathbb Z}^n}}
\def\T{{{\mathbb T}^1}}
\def\N{{{\mathbb N}}}
\def\C{{{\mathbb C}}}
\def\SU2{{{\rm SU(2)}}}
\def\SO3{{{\rm SO(3)}}}
\def\lapsu2{{{\mathcal L}_\SU2}}
\def\LapG{{\mathcal L}_G}

\def\Op{\text{\rm Op}}

\begin{document}

%
\title[Schatten classes  and traces on compact groups]
 {Schatten classes and traces on compact groups}

\author[Julio Delgado]{Julio Delgado}

\address{%
Department of Mathematics\\
Imperial College London\\
180 Queen's Gate, London SW7 2AZ\\
United Kingdom
}

\email{j.delgado@imperial.ac.uk}

\thanks{The first author was supported by Marie Curie IIF 301599 and by the Leverhulme Grant RPG-2014-02. The second author was supported by EPSRC Leadership Fellowship 
EP/G007233/1 and by EPSRC Grant EP/K039407/1. No new data was collected or generated during the course of the research.}
\author[Michael Ruzhansky]{Michael Ruzhansky}

\address{%
Department of Mathematics\\
Imperial College London\\
180 Queen's Gate, London SW7 2AZ\\
United Kingdom
}

\email{m.ruzhansky@imperial.ac.uk}



\date{\today}
\begin{abstract}
In this paper we present symbolic criteria for invariant operators on compact topological groups $G$ characterising the Schatten--von Neumann
classes $S_{r}(L^{2}(G))$ for all $0<r\leq\infty$.     
Since it is known that for pseudo-differential operators criteria in terms of kernels may be
less effective (Carleman's example), our criteria are given in terms of the operators' symbols
defined on the noncommutative analogue of the phase space $G\times\widehat{G}$, where 
$G$ is a compact topological (or Lie) group and $\widehat{G}$ is its unitary dual. 
We also show results concerning general non-invariant operators as well as Schatten
properties on Sobolev spaces.  
A trace formula is derived for operators in the Schatten class $S_{1}(L^{2}(G))$. 
Examples are given for Bessel potentials associated to sub-Laplacians 
(sums of squares) on compact Lie groups, 
as well as for powers of the sub-Laplacian and for other non-elliptic
operators on $\SU2\simeq\mathbb S^3$ and on $\SO3$.
\end{abstract}

\maketitle
\section{Introduction}

Let $G$ be a compact topological (or Lie) group. In this paper we address the problem of 
 characterising the Schatten-von Neumann classes $S_{r}$ of operators on the Hilbert
space $L^{2}(G)$ for $0<r\leq \infty$. To do this, we make use of the global 
quantization recently developed in \cite{rt:book} and \cite{rt:groups} as a noncommutative analogue 
of the Kohn-Nirenberg quantization of operators on $\Rn$. For brevity in the sequel we will refer to 
Schatten-von Neumann classes simply as to Schatten classes.

In view of Carleman's example \cite{car:ex} recalled below, it is well
understood  that there is an obstruction 
when looking for good criteria in terms of kernels for ensuring that an operator belongs
to the Schatten class $S_r$ below the index $r<2$. 

In this paper we show that using the notion of a matrix symbol of an operator on a compact
group, we can characterise or give sufficient conditions for 
operators in Schatten classes $S_r(L^2(G))$ for all $0<r\leq\infty$.
Moreover, our characterisations do not assume any regularity condition on the
symbols. Criteria for operators with symbols in H{\"o}rmander classes have been analysed,
see e.g. Shubin \cite[Section 27]{shubin:r}, but the regularity of symbols is required
for the analysis there. Recently, there was a surge of interest in finding criteria for
Schatten classes in terms of symbols with lower regularity, see e.g.
\cite{Toft:Schatten-AGAG-2006,Toft:Schatten-modulation-2008,
Buzano-Toft:Schatten-Weyl-JFA-2010}.
In this paper we can drop the regularity assumptions completely 
(at least in the considered settings) due to the
technique of matrix quantization that we are using, instead of the traditional
Kohn-Nirenberg quantization in the manifold setting.

 In particular, for operators
acting on $L^{2}(G)$, in Corollary \ref{mainT1} 
we give the characterisation of operators in the Schatten class
$S_{r}(L^{2}(G))$, $0<r\leq\infty$, in terms of symbols of their powers,
and in Theorem \ref{charnucSr} we elaborate this to provide the characterisation
of left-invariant operators in terms of their symbols. 
In particular, the class $S_{\infty}$ of invariant bounded operators on $L^2(G)$ can be
characterised by the uniform boundedness of the operator matrix-norms of their symbol,
see \eqref{EQ:Sinfty}.
Here we note that symbolic criteria for the
$L^{p}$-boundedness of operators on compact Lie groups, the 
Mikhlin-H\"ormander multiplier theorem and its extension to non-invariant
operators for $1<p<\infty$, are presented in
\cite{Ruzhansky-Wirth:Lp-FAA} and in \cite{Ruzhansky-Wirth-multipliers}. 
The nuclearity of pseudo-differential operators on 
the circle $\T$ has been
recently analysed in \cite{dw:trace} but the situation in the present paper is much more
subtle because of the necessarily appearing non-trivial
multiplicities of the eigenvalues of the Laplacian on
the noncommutative compact Lie groups; 
for $0<r\leq 1$, 
the $r$-nuclearity on $L^p$-spaces on compact
Lie groups is addressed in \cite{dr13a:nuclp}.

In order to illustrate the relevance of our symbolic criteria we briefly recall a classical result of 
Carleman (\cite{car:ex}) from 1916, who has constructed a periodic continuous function 
$\varkappa(x)=\sum\limits_{k=-\infty}^{\infty}c_k e^{2\pi ikx}$, i.e. a continuous function on the 
commutative Lie group $\T$, for which the Fourier coefficients $c_k$ satisfy
$$\sum\limits_{k=-\infty}^{\infty} |c_k|^r=\infty\qquad \textrm{ for any } r<2.$$     
Now, considering the normal operator 
\begin{equation}\label{EQ:Carleman}
Tf=f*\varkappa
\end{equation}
acting on $L^2(\T)$ 
one obtains that the sequence $(c_k)_k$ forms a complete system of eigenvalues of 
this operator corresponding to the complete orthonormal system $\phi_k(x)=e^{2\pi ikx}$, 
$T\phi_k=c_k\phi_k$. 
The system $\phi_k$ is also complete for $T^*$, $T^*\phi_k=\overline{c_k}\phi_k$, 
the singular values of $T$ are given by $ s_k(T)=|c_k|$ and
hence
$$\sum\limits_{k=-\infty}^{\infty}s_k(T)^r=\infty ,$$ 
 for $r<2$. Therefore, the invariant operator $T$ is not in $S_r(L^2(\T))$ for $r<2$. However, the continuous integral
 kernel $K(x,y)=\varkappa (x-y)$ satisfies any kind of integral condition of the form 
 $$ \int\int |K(x,y)|^s dxdy<\infty$$ due to the boundedness of $K$. This shows that already in 
 the problem of finding sufficient conditions for an operator to belong to the Schatten 
 class $S_r(L^2(\T))$ with  $r<2$, it is impossible to formulate a sufficient condition of this type for the kernel. 
 On the other hand, it is possible to find such criteria assuming additional regularity of the kernel, the problem
 that was addressed in \cite{dr:suffkernel}, and a comparison between kernel and symbol
 criteria in those cases was done in \cite{dr:sdk}. However, here we are interested in
 exploring criteria with no regularity assumptions on the kernel (and on the symbol).
 
 In this work we will be looking to establish conditions imposed on symbols 
 instead of kernels characterising the operators from Schatten classes. 
We give examples of our results on the torus $\Tn$, on $\SU2$ and on $\SO3$.
Since $\SU2\simeq\mathbb S^3$ are globally diffeomorphic and isomorphic,
with the product of matrices in $\SU2$ corresponding to the quaternionic product
on $\mathbb S^3$, our results immediately extend to $\mathbb S^3$ since the symbolic calculus
is stable under global diffeomorphisms preserving the groups structure
(see \cite[Section 12.5]{rt:book} or \cite{rt:groups}). 
Moreover, in view of the recently resolved Poincar\'e conjecture, any simply-connected closed
3-manifold $M$ is globally diffeomorphic to $\mathbb S^3$ inducing the corresponding group
structure on $M$. Thus, our criteria extend to such $M$ as well with no changes in formulations.

We observe that most of the results of this paper hold true for compact topological groups, without
assuming the differential Lie group structure, which we need only in the statements involving differential operators. We follow \cite[Chapter 7]{rt:book} in assuming that the unit element set $\{e\}$ is closed so that the topological groups are Hausdorff, and also refer to it for other background details on compact topological groups and their representation theory. 

Moreover, we can observe that the notion of matrix-valued symbols on compact groups and our
criteria become instrumental and can be used as a tool in questions which may be formulated 
intrinsically on the group without referring to a particular quantization. 

For example, in Theorem \ref{sobolev},
we conclude (by a simple argument) that Schatten 
classes for left-invariant operators on Sobolev spaces $H^s(G)$
are independent with respect to the order $s$ of the space. 

To give another specific example of conclusions independent of the quantization used,
let $\cL$ and $\cL_{sub}$ denote the (negative-definite) Laplacian and the sub-Laplacian on the group
$G=\SU2$, respectively, (or, with the same conclusion, on the group $\SO3$, or on the quaternionic sphere $\mathbb S^3).$
Then we will show in Section \ref{SEC:examples} that for $0<r<\infty$, we have
$$
(I-\cL)^{-\alpha/2} \in S_r(L^2(\mathbb S^3)) \quad\textrm{ if and only if }\quad \alpha r>3,
$$
while
$$
(I-\cL_{sub})^{-\alpha/2} \in S_r(L^2(\mathbb S^3)) \quad\textrm{ if and only if }\quad \alpha r>4.
$$
Here the powers $(I-\cL_{sub})^{-\alpha/2}$ are well-defined since $\cL_{sub}$ is hypoelliptic,
which follows from H\"ormander's sum of the squares theorem
(see also Greenfield and Wallach \cite{Greenfield-Wallach:hypo-TAMS-1973} for a general framework for
this, and \cite{Ruzhansky-Turunen-Wirth:arxiv} for the associated hypoelliptic symbolic calculus on compact Lie groups).
We also give an example for the family of operators
$$
\mathcal H_{\gamma}=iD_3-\gamma(D_1^2+D_2^2), \; 0<\gamma<\infty,
$$
which are not covered by H\"ormander's sum of the squares theorem. The criterion established in
Theorem \ref{charnucSr} 
 allows one to easily conclude that the operators $(I+\mathcal H_\gamma)^{-\alpha/2}$ are never in
Schatten classes for $0<\gamma\leq 1$, while for $\gamma>1$ show that it is in the class $S_r$
if and only if $\alpha r>4$.

In Section \ref{SEC:Prelim} we make a short introduction to the global quantization on compact groups.
In Section \ref{SEC:nuclearity-2} we give characterisations 
for general operators in Schatten classes
$S_{r}(L^{2}(G))$ for $0<r<\infty$, elaborate this in the case of left-invariant operators,
 and deduce several corollaries, including a criterion for Bessel potentials. 
 In Proposition \ref{PROP:sub-Laplacian} we look at the example of powers of sums of
 squares of vector fields satisfying H\"ormander's commutators condition to a certain order.
 In Section \ref{SEC:examples}
we give examples of our results on the torus $\Tn$, 
and for powers of the Laplacian and sub-Laplacian 
on $\SU2\simeq\mathbb S^3$ and on $\SO3$.
In Section \ref{SEC:traces} we derive a trace formula for left-invariant operators.

The authors would like to thank Jens Wirth for discussions and remarks, and 
Marius M\u{a}ntoiu for a comment.

\section{Preliminaries}  
\label{SEC:Prelim}

In this section we recall some basic facts about the global matrix
quantization on a compact topological/Lie group $G$. 
Let $\widehat{G}$ denote the set of equivalence classes of continuous irreducible unitary 
representations of $G$. Since $G$ is compact, the set $\widehat{G}$ is discrete.  
For $[\xi]\in \widehat{G}$, by choosing a basis in the representation space of $\xi$, we can view 
$\xi$ as a matrix-valued function $\xi:G\rightarrow \ce^{d_{\xi}\times d_{\xi}}$, where 
$d_{\xi}$ is the dimension of the representation space of $\xi$. 
By the Peter-Weyl theorem the collection
$$
\left\{ \sqrt{d_\xi}\,\xi_{ij}: \; 1\leq i,j\leq d_\xi,\; [\xi]\in\Gh \right\}
$$
is the orthonormal basis of $L^2(G)$.
If $f\in L^1(G)$ we define its global Fourier transform at $\xi$ by 
\begin{equation}\label{EQ:FG}
\mathcal F_G f(\xi)\equiv \widehat{f}(\xi):=\int_{G}f(x)\xi(x)^*dx,
\end{equation}
where $dx$ is the normalised Haar measure on $G$. Thus, if $\xi$ is a matrix representation, 
we have $\widehat{f}(\xi)\in\ce^{d_{\xi}\times d_{\xi}} $. The Fourier inversion formula is a consequence
 of the Peter-Weyl theorem, so that we have
\beq \label{EQ:FGsum}
f(x)=\sum\limits_{[\xi]\in \widehat{G}}d_{\xi} \Tr(\xi(x)\widehat{f}(\xi)).
\eq
Given a sequence of matrices $a(\xi)\in\mathbb C^{d_\xi\times d_\xi}$, we can define
\begin{equation}\label{EQ:FGi}
(\mathcal F_G^{-1} a)(x):=\sum\limits_{[\xi]\in \widehat{G}}d_{\xi} \Tr(\xi(x) a(\xi)).
\end{equation}
This series can be interpreted distributionally or absolutely depending on the growth of 
(the Hilbert-Schmidt norms of) $a(\xi)$.
We refer to \cite{rt:book} for further discussion of this background material.
The Parseval identity takes the form 
\begin{equation}\label{EQ:Parseval}
\|f\|_{L^2(G)}= \left(\sum\limits_{[\xi]\in \widehat{G}}d_{\xi}\|\widehat{f}(\xi)\|^2_{\HS}\right)^{1/2},\quad
\textrm{ where }
\|\widehat{f}(\xi)\|^2_{\HS}=\Tr(\widehat{f}(\xi)\widehat{f}(\xi)^*),
\end{equation}
which gives the norm on 
$\ell^2(\widehat{G})$. 

For each $[\xi]\in \widehat{G}$, the matrix elements of $\xi$ are the eigenfunctions for the Laplacian $\mathcal{L}_G$ 
(or the Casimir element of the universal enveloping algebra), with the same eigenvalue which we denote by 
$-\lambda^2_{[\xi]}$, so that we have
\begin{equation}\label{EQ:Lap-lambda}
-\mathcal{L}_G\xi_{ij}(x)=\lambda^2_{[\xi]}\xi_{ij}(x)\qquad\textrm{ for all } 1\leq i,j\leq d_{\xi}.
\end{equation} 
The weight for measuring the decay or growth of Fourier coefficients in this setting is 
$$\jp{\xi}:=(1+\lambda^2_{[\xi]})^{\half},$$ the eigenvalues of the elliptic first-order pseudo-differential operator 
$(I-\mathcal{L}_G)^{\half}$.

For a linear continuous operator $A$ from $C^{\infty}(G)$ to $\mathcal{D}'(G) $ 
we define  its {\em matrix-valued symbol} $\sigma_A(x,\xi)\in\cdxi$ by 
\begin{equation}\label{EQ:A-symbol}
\sigma_A(x,\xi):=\xi(x)^*(A\xi)(x)\in\cdxi,
\end{equation}
where $A\xi(x)\in \cdxi$ is understood as $(A\xi(x))_{ij}=(A\xi_{ij})(x)$, i.e. by 
applying $A$ to each component of the matrix $\xi(x)$.
Then one has (\cite{rt:book}, \cite{rt:groups}) the global quantization
\begin{equation}\label{EQ:A-quant}
Af(x)=\sum\limits_{[\xi]\in \widehat{G}}d_{\xi}\Tr(\xi(x)\sigma_A(x,\xi)\widehat{f}(\xi))
\end{equation}
in the sense of distributions, and the sum is independent of the choice of a representation $\xi$ from each 
equivalence class 
$[\xi]\in \widehat{G}$. If $A$ is a linear continuous operator from $C^{\infty}(G)$ to $C^{\infty}(G)$,
the series \eqref{EQ:A-quant} is absolutely convergent and can be interpreted in the pointwise
sense. We will also write $A=\Op(\sigma_A)$ for the operator $A$ given by
the formula \eqref{EQ:A-quant}. The symbol $\sigma_A$ can be interpreted as a matrix-valued
function on $G\times\widehat{G}$.
We refer to \cite{rt:book, rt:groups} for the consistent development of this quantization
and the corresponding symbolic calculus. 

If the operator $A$ is left-invariant then its symbol
$\sigma_A$ does not depend on $x$. In this case we will often just say that $A$ is invariant.

\section{Schatten classes on $L^2(G)$} 
\label{SEC:nuclearity-2}
In this section we study symbolic criteria for Schatten classes $S_{r}(L^{2}(G))$
in terms of their full matrix-valued symbols, in particular we characterise  
 invariant operators in the Schatten classes $S_{r}(L^{2}(G))$.

We recall that if $A\in\Psi^m(G)$ is a pseudo-differential operators in H\"ormander's class
$\Psi^m(G)$, i.e. if all of its localisations to $\Rn$ are pseudo-differential operators with
symbols in the class $S^m_{1,0}(\Rn)$, then the matrix-symbol of $A$ satisfies
$$\|\sigma_A(x,\xi)\|_{op}\leq C \jp{\xi}^m \qquad \textrm {for all } x\in G,\; [\xi]\in\Gh.$$
Here and everywhere $\|\cdot\|_{op}$ denotes the operator norm of the matrix
multiplication by the matrix $\sigma_A(x,\xi).$
For this fact, see e.g. \cite[Lemma 10.9.1]{rt:book} or \cite{rt:groups}, and for the complete
characterisation of H\"ormander classes $\Psi^m(G)$ in terms of matrix-valued symbols
see also \cite{Ruzhansky-Turunen-Wirth:arxiv}. In particular, this motivates the
appearance of the operator norms of the matrix-valued symbols. However, since $\sigma_{A}$ 
is in general a matrix, other matrix norms become useful as well.

We recall that if $H$ is a complex Hilbert space, 
a linear compact operator $A:H\rightarrow H$ belongs to the Schatten class $S_r(H)$ if 
 $$\sum\limits_{n=1}^{\infty}(s_n(A))^r<\infty,$$ 
where $s_n(A)$ denote the singular values of $A$, i.e. the eigenvalues of $|A|=\sqrt{A^*A}$ 
with multiplicities counted. If $1\leq r<\infty$ 
 the class $S_r(H)$ becomes a Banach space endowed with the norm 
 \[\|A\|_{S_r}=\left(\sum\limits_{n=1}^{\infty}(s_n(A))^r\right)^{\frac{1}{r}}.\]
If $0<r<1$ the identity above only defines a quasi-norm with respect to which $S_r(H)$ is complete. The class $S_2(H)$ and $S_1(H)$ are usually known as the class of Hilbert-Schmidt operators and the trace class, respectively. 

In the case of $r=\infty$ we can put $\|A\|_{S_{\infty}}$ to be the operator norm of the bounded operator
$A:L^{2}(G)\to L^{2}(G)$. In this case it can be easily seen by the Plancherel formula that
for an invariant operator $A$, we have
\begin{equation}\label{EQ:Sinfty}
A\in S_{\infty}(L^{2}(G)) \quad\textrm{ if and only if }\quad \sup_{[\xi]\in\Gh} \|\sigma_{A}(\xi)\|_{op}<\infty,
\end{equation}
see e.g. \cite[Section 10.5]{rt:book}.
So, in the sequel we can assume that $0<r<\infty$.
The following lemma is a consequence of the definition of Schatten classes: 
\begin{lem}\label{lemnucSr2} Let $A:H\to H$ be a linear
compact operator. Let $0<r,t<\infty$. Then $A\in S_r$ if and only if $|A|^{\frac{r}{t}}\in S_t$. Moreover, 
$\|A\|_{S_r}^r=\||A|^{\frac{r}{t}}\|_{S_t}^t$.
\end{lem}
We will denote by $\|\sigma(x,\xi)\|_{S_r}$ the Schatten-norm of order $r$ of the matrix 
$\sigma(x,\xi)\in\C^{\dxi\times\dxi}$, for $x,\xi$ fixed, viewed as a linear mapping
on $\C^{\dxi}$. 

We start by giving a simple criterion for Hilbert-Schmidt operators.
On general $L^2(\mu)$ spaces the Hilbert-Schmidt operators are characterised by the square integrability 
of the kernel with respect to the product measure $\mu\otimes\mu.$ It is also well known that one can translate 
this characterisation in terms of symbols for pseudo-differential operators on the Euclidean space. We state below 
a characterisation of Hilbert-Schmidt operators in terms of the matrix-valued symbol on compact Lie groups.  
This underlines the natural role played by the matrix-valued symbol on compact Lie groups.
The norm of $L^2(G\times\widehat{G})$ can be defined by 
\[\|\sigma\|_{L^2(G\times\widehat{G})}=\left(\int_{G}\sum_{[\xi]\in\Gh}d_{\xi} \|\sigma(x,\xi)\|_{\HS}^2dx\right)^{\half}.\]
This is a natural norm in view of the Parseval identity \eqref{EQ:Parseval}.

\begin{prop}\label{mainhs} 
Let $G$ be a compact Lie group. 
 Let $A:L^{2}(G)\to L^{2}(G)$ be a linear
continuous operator with matrix-valued symbol $\sigma_A(x,\xi)$.
Then $A$ is a Hilbert-Schmidt operator if and only if $\sigma_A\in L^2(G\times\widehat{G}).$ Moreover, $\|A\|_{\HS}=\|\sigma_A\|_{L^2(G\times\widehat{G})}.$ 

In particular, if $\sigma_A(\xi)$
depends only on $\xi$, then
 $A$ is Hilbert-Schmidt if and only if its symbol $\sigma_A$ satisfies 
\beq\sum_{[\xi]\in\Gh}d_{\xi}\|\sigma_{A}(\xi)\|_{\HS}^2<\infty.\eq
\end{prop}
\begin{proof}
The kernel of $A$ is given by  
\beq K(x,y)=\sum_{[\xi]\in\widehat{G} }d_{\xi} \Tr(\xi(x)\sigma_A(x,\xi)\xi(y)^* ).\label{kernel1}\eq
We have $\|A\|_{\HS}^2=\int_G\int_G |K(x,y)|^2dxdy=\int_G\int_G |K(x,xz^{-1})|^2dxdz.$
From (\ref{kernel1}) we obtain 
\begin{align*}
K(x,xz^{-1})=& \sum_{[\xi]\in\widehat{G} }d_{\xi} \Tr(\xi(x)\sigma_A(x,\xi) \xi(xz^{-1})^* )\\
=& \sum_{[\xi]\in\widehat{G} }d_{\xi} \Tr(\xi(xz^{-1})^*\xi(x)\sigma_A(x,\xi) )\\
=& \sum_{[\xi]\in\widehat{G} }d_{\xi} \Tr(\xi(z)\sigma_A(x,\xi) )\\
=& \mathcal{F}_{G}^{-1}\sigma_A(x,\cdot)(z),
\end{align*}
with $\mathcal{F}_{G}^{-1}$ defined in \eqref{EQ:FGi}.
We have formally
\begin{multline*}
\int_G\int_G |K(x,xz^{-1})|^2dxdz= \int_G\int_G |\mathcal{F}_{G}^{-1}\sigma_A(x,\cdot)(z)|^2dzdx\\
= \int_{G}\sum_{[\xi]\in\Gh}d_{\xi} \|\sigma_A(x,\xi)\|_{\HS}^2dx
= \|\sigma_A\|_{L^2(G\times\widehat{G})}^2.
\end{multline*}
The second equality is obtained from the Parseval identity \eqref{EQ:Parseval}.
\end{proof}

As a consequence of Lemma \ref{lemnucSr2} with $t=2$, and 
Proposition \ref{mainhs} we have:
\begin{cor}\label{mainT1} \label{mainT1a} 
Let $G$ be a compact Lie group. 
 Let $A:L^{2}(G)\to L^{2}(G)$ be a linear
continuous operator with matrix-valued symbol $\sigma_A(x,\xi)$. Let $0<r<\infty $. 
Then $A\in S_r(L^{2}(G))$ if and only if 
$$\sum_{[\xi]\in\Gh}d_{\xi}\int_{G} \|\sigma_{|A|^{\frac{r}{2}}}(x,\xi)\|_{\HS}^2dx<\infty.$$
In particular, for invariant operators, i.e. for operators $A$ with matrix-valued symbols $\sigma_A(\xi)$ depending only on $\xi$, we have that  $A\in S_r(L^{2}(G))$ if and only if 
$$\sum_{[\xi]\in\Gh}d_{\xi}\|\sigma_{|A|^{\frac{r}{2}}}(\xi)\|_{\HS}^2<\infty.$$
\end{cor}
 
With the purpose of characterising invariant operators
belonging to a Schatten class $S_r(L^2(G))$ for $0<r<\infty$ we first establish 
 a simple fact required for the characterisation  in the case of invariant operators:

\begin{lem}\label{lemnucSr} 
Let $A:L^{2}(G)\to L^{2}(G)$ be a linear
continuous operator with matrix-valued symbol $\sigma_A(\xi)$
depending only on $\xi.$ Then, for each $[\xi]\in\Gh$ and $0<r<\infty $ we have
\[\sigma_{|A|^r}(\xi)=|\sigma_{A}(\xi)|^r.\]
\end{lem}
\begin{proof} Since $A$ is invariant, we have
\[\sigma_{A^*A}(\xi)=\sigma_{A^*}(\xi)\sigma_{A}(\xi),\]
which can be expressed as $\sigma_{|A|^{2}}(\xi)=|\sigma_{A}(\xi)|^{2}.$
In general, since the operator $|A|$ is formally self-adjoint, its symbol $\sigma_{|A|}(\xi)$ is 
self-adjoint for every $\xi$; indeed, we have 
$\sigma_{|A|}(\xi)=\sigma_{|A|^{*}}(\xi)=\sigma_{|A|}(\xi)^{*}$, 
the last equality because of the left-invariance of $|A|$. Consequently, we can diagonalise
$\sigma_{|A|}(\xi)$ by a unitary transformation, which means that we can assume that
$\sigma_{|A|}(\eta)$ is diagonal for some $\eta\in[\xi]$. Consequently, the positivity
of the symbol (of the positive operator $|A|)$ implies that 
$\sigma_{|A|^r}(\eta)=|\sigma_{A}(\eta)|^r$ for any $0<r<\infty$. Going back to the
representation $\xi$ we obtain the statement.
\end{proof}

We can now state the characterisation for invariant operators in $S_r(L^2(G))$. 
We note here that since it is formulated for invariant operators, we do not need the
differential structure on $G$, so that all the arguments in the proof actually work in the
more general setting of compact topological groups.

\begin{thm}\label{charnucSr} 
Let $G$ be a compact topological group and let $0<r<\infty$. 
 Let $A:L^{2}(G)\to L^{2}(G)$ be a linear
compact operator with matrix-valued symbol $\sigma_A(\xi)$
depending only on $\xi.$  Then
$A\in S_r(L^2(G))$ if and only if  
$$\sum_{[\xi]\in\Gh}d_{\xi} \|\sigma_A(\xi)\|_{S_r}^r<\infty.$$ 
\end{thm}
\begin{proof} By Corollary \ref{mainT1a} $A\in S_r(L^2(G))$ if and only if $\sum_{[\xi]\in\Gh}d_{\xi} \|\sigma_{|A|^{\frac r2}}(\xi)\|_{S_2}^2<\infty$. But Lemma \ref{lemnucSr}  gives us $\|\sigma_{|A|^{\frac r2}}(\xi)\|_{S_2}=\||\sigma_{A}(\xi)|^{\frac r2}\|_{S_2}$. From Lemma \ref{lemnucSr2} applied to the matrix-symbol $\sigma_A(\xi)$ for each $\xi$, we obtain $\||\sigma_{A}(\xi)|^{\frac r2}\|_{S_2}^2=\||\sigma_{A}(\xi)|^{r}\|_{S_1}$. Applying Lemma \ref{lemnucSr2} once more gives us $\||\sigma_{A}(\xi)|^{r}\|_{S_1}=\|\sigma_{A}(\xi)\|_{S_r}^r$, concluding the proof.



\end{proof}

\begin{rem} The expression $\sum_{[\xi]\in\Gh}d_{\xi} \|\sigma_A(\xi)\|_{S_r}^r$ appearing
 in the condition of Theorem \ref{charnucSr} characterising invariant operators in 
 Schatten classes comes from the norm 
\begin{equation}\label{EQ:norms-sch}
\|\sigma_{A}\|_{\ell^{r}_{sch}}:=
\left(\sum_{[\xi]\in\Gh}d_{\xi} \|\sigma_A(\xi)\|_{S_r}^r\right)^{1/r}.
\end{equation}
For the analysis of spaces 
with norms $\|\sigma_{A}\|_{\ell^{r}_{sch}}$ we refer
to Hewitt and Ross \cite[Section 31]{Hewitt-Ross:BK-Vol-II} or 
Edwards \cite[Section 2.14]{Edwards:BK}.
Because of the Hausdorff-Young inequality for the Fourier transform in these spaces 
(see Kunze \cite{Kunze:FT-TAMS-1958})
we get the following one-sided criteria for invariant operators:
\end{rem}

\begin{cor}\label{COR:kernel-schatten}
Let $G$ be a compact Lie group and 
let the left-invariant operator $A$ be bounded on $L^{2}(G)$.
Then it is of the form $Af=f*k$ with $k\in \mathcal D'(G)$ such that
$\sup_{[\xi]\in\Gh}\|\widehat{k}(\xi)\|_{op}<\infty$. Moreover, we have the following
properties:
\begin{itemize}
\item[(i)] if $k\in L^{p}(G)$ with $1\leq p\leq 2$, then $A\in S_{p'}(L^{2}(G))$ with 
$\frac1p+\frac{1}{p'}=1;$
\item[(ii)] if $A\in S_{p}(L^{2}(G))$ with $1\leq p\leq 2$, then $k\in L^{p'}(G)$ with 
$\frac1p+\frac{1}{p'}=1.$
\end{itemize}
\end{cor}
\begin{proof}
The representation $Af=f*k$ follows from the Schwartz kernel theorem, while the $L^{2}$-boundedness
implies that $\sup_{[\xi]\in\Gh}\|\widehat{k}(\xi)\|_{op}<\infty$ by \eqref{EQ:Sinfty}.
We also have $\sigma_{A}(\xi)=\widehat{k}(\xi).$ Consequently, both properties 
(i) and (ii) follow from Theorem \ref{charnucSr} and
 from the Hausdorff-Young inequalities for the Fourier transform
mapping between $L^{p}(G)$ and spaces with the norm
$$\|\widehat{f}\|_{\ell^{p'}_{sch}}:=
\left(\sum_{[\xi]\in\Gh}d_{\xi} \|\widehat{f}(\xi)\|_{S_{p'}}^{p'}\right)^{1/{p'}},$$
see also \eqref{EQ:norms-sch}.
\end{proof}

The following lemma has been proved in \cite{Dasgupta-Ruzhansky:BSM} and it will be useful here to deduce other consequences.
\begin{lem}\label{lem2} 
Let $G$ be a compact Lie group. Then we have 
$$\sum_{[\xi]\in \widehat{G}}d_{\xi}^2\jp{\xi}^{-s}<\infty
\quad\textrm{ if and only if }\quad s>\dim G.$$ 
\end{lem}

This Lemma \ref{lem2} and Theorem \ref{charnucSr} yield the following corollary:
\begin{cor}\label{corschr1} 
Let $G$ be a compact Lie group of dimension $n$ and let $A$ be an operator with symbol
$\sigma(\xi)$. Let $0<r<\infty$. Suppose that 
$\|\sigma(\xi)\|_{S_r}\leq Cd_{\xi}^{1/r}\jp{\xi}^{-\frac{s}{r}}$ 
with some $s>n$. Then  $A\in S_r(L^{2}(G))$.
\end{cor}

Before we give several examples of this corollary in Section \ref{SEC:examples},
let us apply it to the Bessel potentials on $G$, also showing the sharpness of the
obtained orders. As before, we denote by $\LapG$ a Laplacian on $G$,
see Stein \cite{Stein:BOOK-topics-Littlewood-Paley} for a general discussion on
Laplacians on compact Lie groups. We recall that  $\LapG$ is a negative-definite
second order bi-invariant elliptic differential operators, the eigenvalues of 
$(I-\LapG)^{1/2}$ are denoted by $\jp{\xi}$, and the symbol of its powers is
$$
\sigma_{(I-\LapG)^{-\alpha/2}}(\xi)=\jp{\xi}^{-\alpha} I_{\dxi},
$$
where $I_{\dxi}\in\C^{\dxi\times\dxi}$ is the identity matrix.
In this case we then readily calculate
$\|\sigma_{(I-\LapG)^{-\alpha/2}}(\xi)\|_{S_r}=d_\xi^{1/r}\jp{\xi}^{-\alpha}.$
In particular, this, together with the following proposition, shows that the orders in 
Corollary \ref{corschr1} are sharp.

\begin{prop}\label{PROP:bessels} 
Let $G$ be a compact Lie group of dimension $n$. 
Then  $(I-\LapG)^{-\alpha/2}$ is in the Schatten class $S_r(L^{2}(G))$, $0<r<\infty$,
if and only if $\alpha r>n$.
\end{prop}
\begin{proof}
We have $\|\sigma_{(I-\LapG)^{-\alpha/2}}(\xi)\|_{S_r}=d_\xi^{1/r}\jp{\xi}^{-\alpha}$, so that
$$
\sum_{[\xi]\in\Gh}d_{\xi} \|\sigma_{(I-\LapG)^{-\alpha/2}}(\xi)\|_{S_r}^r=
\sum_{[\xi]\in\Gh}d_{\xi}^2 \jp{\xi}^{-\alpha r},
$$
and Proposition \ref{PROP:bessels} follows from combining the criteria in Theorem
\ref{charnucSr}  and in Lemma \ref{lem2}.
\end{proof}
We note that the statement of Proposition \ref{PROP:bessels} 
can be extended to a more general setting
on compact manifolds, see \cite{Delgado-Ruzhansky:jdam}, as well as \cite{dr:suffkernel} for a comparison,
so the more interesting setting for us here is that of non-elliptic
operators that we address in the next section.

We will now give some consequences regarding Schatten classes on Sobolev spaces. 
The notion of global symbol and the characterisation given by Theorem \ref{charnucSr} of Schatten 
classes for left-invariant operators 
provide a simple proof of the independence with respect to the scale of Sobolev spaces
$H^{s}(G)$ (defined by their localisations being in $H^{s}(\Rn)$). 

\begin{thm}\label{sobolev} 
Let $G$ be a compact Lie group, $0<r<\infty$ and $ s\in\er$. 
 Let $A:H^{s}(G)\to H^{s}(G)$ be a linear
bounded operator with the matrix-valued symbol $\sigma_A(\xi)$
depending only on $\xi.$  Then
\[A\in S_r(H^s(G)) \mbox{ if and only if }  \sum_{[\xi]\in\Gh}d_{\xi} \|\sigma_A(\xi)\|_{S_r}^r<\infty.\]
Consequently, if $ s_1,s_2\in\er$ and $A$ is
a left-invariant operator bounded on 
$A:H^{s_1}(G)\to H^{s_1}(G)$ and $A:H^{s_2}(G)\to H^{s_2}(G)$, then 
\[A\in S_r(H^{s_1}(G)) \mbox{ if and only if }  A\in S_r(H^{s_2}(G)).\]
\end{thm}

\begin{proof} We observe that if $Au=\lambda u$ with $u\in H^s(G)$, writing 
\[u=(I-\cL)^{-\frac{s}{2}}f,\]
with $f\in L^2(G)$ and $\cL$ denoting a Laplacian on $G$, we get
\[A(I-\cL)^{-\frac{s}{2}}f=\lambda (I-\cL)^{-\frac{s}{2}}f.\]
Hence we have
\[\tilde{A}f=\lambda f,\, \textrm{ with } \, \tilde{A}=(I-\cL)^{\frac{s}{2}}A(I-\cL)^{-\frac{s}{2}}.\]
Thus $\tilde{A}:L^2(G)\rightarrow L^2(G)$ is bounded on $L^2(G)$ and
\begin{align*}
\sigma_{\tilde{A}}(\xi)=\sigma_{(I-\cL)^{\frac{s}{2}}}(\xi)\sigma_{A}(\xi)\sigma_{(I-\cL)^{-\frac{s}{2}}}(\xi)
=\sigma_{A}(\xi)
\end{align*}
because all these operators are left-invariant and 
$\sigma_{(I-\cL)^{\frac{s}{2}}}(\xi)=\jp{\xi}^{s} I_{\d_\xi}$ is diagonal.
Therefore, using Theorem \ref{charnucSr}, we obtain
\begin{align*}
A\in S_r(H^s(G))\Longleftrightarrow &\tilde{A}\in S_r(L^2(G))\\
\Longleftrightarrow &\sum_{[\xi]\in\Gh}d_{\xi} \|\sigma_{\tilde{A}}(\xi)\|_{S_r}^r<\infty\\
\Longleftrightarrow & \sum_{[\xi]\in\Gh}d_{\xi} \|\sigma_{{A}}(\xi)\|_{S_r}^r<\infty,
\end{align*}
completing the proof.
\end{proof}




Similar to Proposition \ref{PROP:bessels} we can give a condition for powers of 
sub-Laplacians. Let $G$ be a compact Lie group of dimension $n$ and let  
\[\subll=X_1^2+\cdots +X_k^2\]
be the sum of squares of left-invariant vector fields $X_1,\ldots,X_k$, for which we assume
that the H\"ormander commutator condition is satisfied of order $\varkappa\in\mathbb N$, i.e.
the commutators of length $\varkappa$ span the Lie algebra of $G$.
Since the operator $\subll$ is formally self-adjoint, we can choose the bases in the representation
spaces in such a way that its symbol is diagonal, and we denote
$$\sigma_{I-\subll}(\xi)={\rm diag}\{\nu_{1}^{2}(\xi),\ldots,\nu_{d_{\xi}}^{2}(\xi)\},$$
for some $\nu_{j}(\xi)\geq 0$. Then, following a-priori estimates by Rothschild and
Stein \cite{Rothschild-Stein:AM-1976}, it was shown in 
\cite[Proposition 3.1]{Garetto-Ruzhansky:JDE2} that
there exist constants $c,C>0$ such that
\begin{equation}\label{EQ:sym-subl}
c\jp{\xi}^{1/\varkappa}\leq \nu_{j}(\xi)\leq C\jp{\xi}
\end{equation}
holds for all $[\xi]\in\Gh$ and all $1\leq j\leq d_{\xi}.$

\begin{prop}\label{PROP:sub-Laplacian} 
Let $G$ be a compact Lie group of dimension $n$. 
Let $$\subll=X_1^2+\cdots +X_k^2$$
be the sum of squares of left-invariant vector fields $X_1,\ldots,X_k$ satisfying 
the H\"ormander commutator condition of order $\varkappa$.
Then  $(I-\mathcal L_{sub})^{-\alpha/2}$ is in the Schatten class $S_r(L^{2}(G))$, $0<r<\infty$,
provided that $\alpha r>\varkappa n$.
\end{prop}
\begin{proof}
The statement follows readily from Theorem \ref{charnucSr} using \eqref{EQ:sym-subl}.
Namely, we can write
\begin{equation}\label{EQ:sub-1}
\sum_{[\xi]\in\widehat{G}}d_{\xi} \|\sigma_{(I-\subll)^{-\frac{\alpha}{2}}}(\xi)\|_{S_r}^r
=
\sum_{[\xi]\in\widehat{G}}\sum\limits_{j=1}^{d_{\xi}}d_{\xi}
\nu_{j}(\xi)^{-\alpha r} 
 \leq
C \sum_{[\xi]\in\widehat{G}} d_{\xi}^2 \langle\xi\rangle^{-\frac{\alpha r}{\varkappa}}.
\end{equation}
Therefore,
\eqref{EQ:sub-1} is finite for $\alpha r>\varkappa n$ in view of Lemma
\ref{lem2}.
\end{proof}

For the case of the Laplacian, we have $\varkappa=1$, so that we recover 
the condition in Proposition \ref{PROP:bessels}. For $\varkappa\geq 2$, however,
the condition $\alpha r>\varkappa n$ in Proposition \ref{PROP:sub-Laplacian}
can be improved in a number of cases.
This requires a more careful analysis of the maximal weight lattice and will be addressed 
elsewhere. In Corollary \ref{COR:su2-subLap}, following explicit calculations,
we give its improvement in the case of the group $\SU2$.

Assuming $\varkappa=2$ for simplicity, we
note that the corresponding condition $\alpha r>2n$ in Proposition \ref{PROP:sub-Laplacian}
is related to subelliptic estimates for the sub-Laplacian that we now briefly indicate.
Indeed, since $\subll$ is a sum of squares associated to a system of vector fields satisfying 
the H{\"o}rmander condition of order $\varkappa=2$, we have
\[\|u\|_{H^{1}}\leq C\|(I-\subll)u\|_{L^2},\]
see e.g. Rothschild and Stein \cite{Rothschild-Stein:AM-1976}.
It follows then that 
$\|(I-\subll)^{-1}u\|_{H^1}\leq C\|u\|_{L^2}$, and hence, by interpolation,
\begin{equation}\label{EQ:subell}
\|(I-\subll)^{-s}u\|_{H^{s}}\leq C\|u\|_{L^2},
\end{equation}
for any $s\geq 0$. Before we apply this, 
we observe that on the other hand, if $\delta$ is the delta-distribution at the unit element of the group, we have
\begin{align*} \efeeg((I-\subll)^{-\frac{\beta}{2}}\delta )(\xi)=\sigma_{(I-\subll)^{-\frac{\beta}{2}}}(\xi)\widehat{\delta}(\xi)
={\rm diag}\{\nu_{j}(\xi)^{-\beta}\}I_{d_{\xi}}.
\end{align*}
Taking $\beta=\frac{\alpha r}{2}$, we get
\begin{align*} 
\|\efeeg((I-\subll)^{-\frac{\alpha r}{4}}\delta )(\xi)\|_{\ell^2(\widehat{G})}^2=&
\sum_{[\xi]\in\widehat{G}}d_{\xi}\|\efeeg((I-\subll)^{-\frac{\alpha r}{4}}\delta )(\xi)\|_{\HS}^2\\
=& \sum_{[\xi]\in\widehat{G}}\sum\limits_{j=1}^{d_{\xi}}d_{\xi}
\nu_{j}(\xi)^{-\alpha r} .
\end{align*}
By Plancherel Theorem, and combining this with the equality in 
\eqref{EQ:sub-1}, we get
\[
\sum_{[\xi]\in\widehat{G}}d_{\xi} \|\sigma_{(I-\subll)^{-\frac{\alpha}{2}}}(\xi)\|_{S_r}^r=
\|(I-\subll)^{-\frac{\alpha r}{4}}\delta \|_{L^2({G})}^2.\]
Since the Laplacian and the sub-Laplacian commute, using
\eqref{EQ:subell},
we get
\[\|(I-\subll)^{-\frac{\alpha r}{4}}\delta\|_{L^2}\leq C\|\delta\|_{H^{-\frac{\alpha r}{4}}}<\infty\]
for $\frac{\alpha r}{4}>\frac{n}{2},$ i.e. for $\alpha r>2n$, the same order as in Proposition
\ref{PROP:sub-Laplacian}.
Although the sub-elliptic estimate \eqref{EQ:subell} is sharp,
the possibility of improving the orders for Schatten classes, such as the one
that we obtain in Corollary \ref{COR:su2-subLap}, can be explained
by the fact that we only need to apply \eqref{EQ:subell} to the delta-distribution, in which case the
Sobolev order can be actually better.

\section{Examples on $\Tn$, $\SU2\simeq\mathbb S^3$ and $\SO3$}
\label{SEC:examples}

We now give some examples of our results on the torus $\Tn$, $\SU2$ and $\SO3$. 
In particular, this shows that the notion of the matrix-valued symbol
becomes instrumental and can be used as a tool for deriving properties of 
operators defined intrinsically on the group.

\subsection{The torus  $\Tn$}
We start with a few simple observations in the case of the torus.
If $G=\Tn=\Rn/ \Zn$, we have $\widehat{\Tn}\simeq \Zn$, and the collection $\{ \xi_k(x)=e^{2\pi i x\cdot k}\}_{k\in\Zn}$ 
is the orthonormal basis of $L^2(\Tn)$, and all $d_{\xi_k}=1$. 
If an operator $A$ is invariant on $\Tn$, its symbol becomes
$\sigma_A(\xi_k)=\xi_k(x)^* A\xi_k(x)=A\xi_k(0).$ 
In general, on the torus we will often simplify the notation by identifying $\widehat{\Tn}$ with $\Zn$,
and thus writing $\xi\in\Zn$ instead of $\xi_k\in\Zn$. The toroidal quantization 
\begin{equation}\label{EQ:A-torus}
Af(x)=\sum\limits_{\xi\in\Zn} e^{2\pi i x\cdot \xi} \sigma_A(x,\xi)\widehat{f}(\xi)
\end{equation}
has been analysed extensively in \cite{Ruzhansky-Turunen-JFAA-torus} and it 
is a special case of \eqref{EQ:A-quant}, where we have identified, as noted,
$\widehat{\Tn}$ with $\Zn$.
As a consequence of Theorem \ref{charnucSr} and Corollary \ref{corschr1} on the torus, we obtain:
\begin{cor}\label{cor125} 
Let $A:L^{2}(\tn)\to L^{2}(\tn)$ be a linear
continuous operator with symbol $\sigma_A(\xi)$
depending only on $\xi.$ Let $0<r<\infty$, then
 $A$ belongs to $S_r(L^{2}(\tn))$ if and only if its symbol $\sigma_A$ satisfies 
\beq\label{EQ:cond-torus}
\sum_{\xi\in\zn} |\sigma_A(\xi)|^r<\infty.
\eq
In particular,
 $A$ belongs to $S_r(L^{2}(\tn))$ provided that 
\beq\label{EQ:cond-torus2}
|\sigma_A(\xi)|\leq C\jp{\xi}^{-\frac{s}{r}},
\eq for some $s>n$.
\end{cor}

\begin{rem} Corollary \ref{cor125} implies that a necessary condition for nuclearity ($r=1$) 
on $L^{2}(\tn)$ for operators with symbol only depending on $\xi$  
is the continuity of the corresponding kernel. Indeed, since 
\[K(x,y)=\sum_{\xi\in\zn}e^{i(x-y)\xi}\sigma_A(\xi),\]
the continuity of $K$ follows from the fact  that $\sigma_A\in \ell^1(\zn)$.  
An analogue of this property on general $G$ was given in 
Corollary \ref{COR:kernel-schatten}.
\end{rem}

The result concerning the Bessel potentials in Proposition
\ref{PROP:bessels} in the case of $S_r(L^2(\Tn))$ becomes as follows:

\begin{prop}\label{PROP:torus}
Let $\Delta$ be the Laplacian on the torus $\Tn$ and let $0<r<\infty$. Then
$(I-\Delta)^{-\frac{\alpha}{2}}$ belongs to  $S_r(L^2(\Tn))$ if and only if $\alpha r>n.$
\end{prop}

\begin{proof} 
We give a direct simple proof of this.
The symbol of the operator $T=(I-\Delta)^{-\frac{\alpha}{2}}$ is positive, 
hence $T$ being a multiplier operator, it is positive definite and 
 $|T|=\sqrt{T^*T}=T$. Thus, the singular values of $T$ agree with the values of its symbol 
 $ \jp{\xi}^{-\alpha}$. Therefore, $T\in S_r(L^2(\Tn)$ if and only if $\alpha r>n.$
\end{proof}

\subsection{The groups $\SU2\simeq\mathbb S^3$ and $\SO3$}
We now consider the case of the non\-commutative group $G=\SU2$, the group
of the unitary $2\times 2$ matrices of determinant one. 
The same results as given below hold for the 3-sphere $\mathbb S^{3}$ if we use
the identification $\SU2\simeq\mathbb S^{3}$, with the matrix multiplication in 
$\SU2$ corresponding to the quaternionic product on $\mathbb S^{3}$,
with the corresponding identification of the symbolic calculus, see
\cite[Section 12.5]{rt:book}. Consequently, the results below extend to any simply-connected
closed 3-manifold in view of the resolved Poincar\'e conjecture, see the discussion in
the introduction. More generally, they can be extended to general closed manifolds, see \cite{Delgado-Ruzhansky:jdam}.

The details of the global quantization
\eqref{EQ:A-quant} on $\SU2$ have been worked out in
\cite[Chapter 12]{rt:book}, to which we 
also refer for the details on the representation theory of the group $G=\SU2$.
In this case, we can enumerate the elements of its dual as $\widehat{G}\simeq \frac12 \N_0$, 
with $\N_0=\{0\}\cup\mathbb N$, so that
$$
\widehat{\SU2}=\{ [t^\ell]: t^\ell\in\C^{(2\ell+1)\times (2\ell+1)},  \ell\in \frac12 \N_0\}.
$$
The dimension of each $t^\ell$ is $d_{t^\ell}=2\ell+1$, and there are explicit formulae for $t^\ell$ 
as functions of
Euler angles in terms of the so-called Legendre-Jacobi polynomials, see
\cite[Chapter 11]{rt:book}. The Laplacian on $\SU2$ has eigenvalues
$\lambda_{t^\ell}^2=\ell(\ell+1)$, so that we have $\jp{t^\ell}\approx \ell$. 
If $A:L^{2}(\SU2)\to L^{2}(\SU2)$ is a continuous linear  operator, its matrix-symbol is denfined by  
$$
\sigma_A(x,\ell)\equiv \sigma_A(x,t^\ell):=t^\ell(x)^* At^\ell(x),\quad \ell\in\frac12\N_0.
$$
Corollary \ref{corschr1}  in this case becomes:

\begin{cor}\label{COR:SU2} 
Let $A:L^{2}(\SU2)\to L^{2}(\SU2)$ be an invariant operator with matrix symbol  
$\sigma_A(\ell)$. Let $s>3$ and let $0<r<\infty$. If there is a constant $C>0$ such that 
$$
\|\sigma_A(\ell)\|_{S_r}\leq C \ell^{\frac{1-s}{r}}
$$ 
for all $\ell\in\frac12\N$, then 
$A\in S_r(L^{2}(\SU2))$.
\end{cor}
We now discuss examples of two operators with diagonal symbols, the Laplacian and the
sub-Laplacian.

If $\lapsu2$ denotes the Laplacian on $\SU2$, 
we have $\lapsu2 t^{\ell}_{mn}(x)=-\ell(\ell+1) t^{\ell}_{mn}(x)$ for all $\ell,m,n$ and $x\in G$,
so that the symbol of $I-\lapsu2$ is given by
$$\sigma_{I-\lapsu2}(x,\ell)=(1+\ell(\ell+1)) I_{2\ell+1},$$ 
where
$I_{2\ell+1}\in \C^{(2\ell+1)\times (2\ell+1)}$ is the identity matrix.
 Hence, $\sigma_{I-\lapsu2}(x,\ell)$ is diagonal and independent of $x$. 
 Consequently, for $(I-\lapsu2)^{-\frac\alpha{2}}$ we have 
  $\|\sigma_{(I-\lapsu2)^{-\frac\alpha{2}}}\|_{S_r}\approx \ell^{-\alpha}\ell^{\frac{1}{r}}$. 
 Therefore, by Corollary \ref{COR:SU2},  
   $(I-\lapsu2)^{-\frac\alpha{2}}$ is in $S_r(L^2(\SU2))$ provided that 
$\ell^{-\alpha}\leq C \ell^{-\frac{s}{r}}$ for $s>3$, which agrees 
with Proposition \ref{PROP:bessels}:

\begin{cor}\label{COR:su2-Lap}
Let $0<r<\infty$. Then the operator
$(I-\lapsu2)^{-\frac{\alpha}{2}}$ is in $S_r(L^2(\SU2))$ if and only if
$\alpha> \frac{3}{r}$. 
\end{cor}
 
To give a slightly different example, 
we shall now consider the group $\SO3$ of the $3\times 3$ real orthogonal matrices 
of determinant one. For the details of the representation theory and the global 
quantization of $\SO3$ we refer the reader to \cite[Chapter 12]{rt:book}.
 The dual in this case can be identified as $\widehat{G}\simeq \ene_0$, so that
$$
\widehat{\SO3}=\{ [t^\ell]: t^\ell\in\C^{(2\ell+1)\times (2\ell+1)},  \ell\in \N_0\}.
$$
The dimension of each $t^\ell$ is $d_{t^\ell}=2\ell+1$. The Laplacian on $\SO3$ has eigenvalues
$\lambda_{t^\ell}^2=\ell(\ell+1)$, so that we have $\jp{t^\ell}\approx \ell$.
By the same argument as above, Corollary \ref{COR:su2-Lap}
also holds for the Laplacian on $\SO3$.

Let us fix three invariant vector fields $D_1, D_2, D_3$ on $\SO3$ 
corresponding to the derivatives with
respect to the Euler angles. We refer to \cite[Chapter 11]{rt:book} for the explicit formulae for these.
However, for our purposes here we note that the sub-Laplacian $\mathcal L_{sub}=D_1^2+D_2^2$,
with an appropriate choice of basis in the representation spaces, has the diagonal symbol given by
\begin{equation}\label{EQ:SO3-subLap}
\sigma_{\mathcal L_{sub}}(\ell)_{mn}=(m^2-\ell(\ell+1))\delta_{mn},\quad m,n\in\mathbb Z,\;
-\ell\leq m,n\leq\ell,
\end{equation}
where $\delta_{mn}$ is the Kronecker delta, and where it is customary to let $m,n$ run from
$-\ell$ to $\ell$ rather than from $0$ to $2\ell+1$.
The operator $\mathcal L_{sub}$ is a second order hypoelliptic operator and we can define the
powers $(I-\mathcal L_{sub})^{-\alpha/2}$. These are pseudo-differential operators with symbols
$$\sigma_{(I-\mathcal L_{sub})^{-\alpha/2}}(\ell)_{mn}=(1+\ell(\ell+1)-m^2)^{-\alpha/2}\delta_{mn},$$
with $m,n\in\mathbb Z,\;
-\ell\leq m,n\leq\ell$.
We now have 
$$
\|\sigma_{(I-\mathcal L_{sub})^{-\alpha/2}}(\ell)\|_{S_r}=
\left(\Tr (\sigma_{(I-\mathcal L_{sub})^{-\alpha/2}}(\ell))^r\right)^{\frac 1r}=
\left(\sum_{m=-\ell}^\ell\left(1+\ell(\ell+1) - m^2\right)^{-\frac{\alpha r}{2}}\right)^{\frac 1r},
$$
where $\ell\in\N_0$. Comparing with the integral 
$$
\int_{-R}^R(1+R^2-x^2)^{-\frac{\alpha r}{2}}dx\approx C R^{-\frac{\alpha r}{2}}
\int_0^R(1+R-x)^{-\frac{\alpha r}{2}}dx\approx C R^{-\frac{\alpha r}{2}},
$$
for $\alpha r>2$ and large $R$,
it follows that $\sum_{m=-\ell}^\ell\left(1+\ell(\ell+1) - m^2\right)^{-\frac{\alpha r}{2}}$ 
is of order $\ell^{-\frac{\alpha r}{2}}$. Hence, $\|\sigma_{(I-\mathcal L_{sub})^{-\alpha/2}}(\ell)\|_{S_r}$
is of order $\ell^{-\frac{\alpha}{2}}$. Therefore, we have
 $$
 \sum_{[\xi]\in\widehat{\SO3}}d_{\xi} \|\sigma_{(I-\mathcal L_{sub})^{-\alpha/2}}(\xi)\|_{S_r}^r\approx
 C \sum_{\ell\in\N} \ell^{1-\frac{\alpha r}{2}},
 $$
and as a consequence of Theorem \ref{charnucSr}, we obtain

\begin{cor}\label{COR:su2-subLap}
Let $0<r<\infty$. Then the operator
$(I-\mathcal L_{sub})^{-\frac{\alpha}{2}}$ belongs to the Schatten class $S_r(L^2(\SO3))$ if and only if
$\alpha> \frac 4r$.
The same conclusion holds if we replace $\SO3$ by 
$\SU2$ or by $\mathbb S^3$ (with a quaternionic sub-Laplacian).
\end{cor}


We also present another example of an operator (on $\SO3$)
which is not covered by H\"ormander's sum of squares theorem.
Namely, we consider the following family of `Schr{\"o}dinger operators'
\[\mathcal H_{\gamma}=iD_3-\gamma(D_1^2+D_2^2),\]
for a parameter $0<\gamma<\infty.$ For $\gamma=1$ it 
was shown in \cite{Ruzhansky-Turunen-Wirth:arxiv} that
$\mathcal H_1+cI$ is globally hypoelliptic if and only if
$0\not\in\{c+\ell(\ell+1)-m(m+1): \ell\in\N, m\in\mathbb Z, |m|\leq\ell\}.$ 

The matrix-symbol of $I+\mathcal H_{\gamma}$ is given by 
\begin{equation}\label{EQ:SO3-schor}
\sigma_{I+\mathcal H_{\gamma}}(\ell)_{mn}=(1+m-\gamma m^2+\gamma\ell(\ell+1))\delta_{mn},\quad m,n\in\mathbb Z,\;
-\ell\leq m,n\leq\ell,
\end{equation}
where as before $\delta_{mn}$ is the Kronecker delta, and we let $m,n$ run from
$-\ell$ to $\ell$ rather than from $0$ to $2\ell+1$.
Similarly to the case of $\gamma=1$ above, for $\gamma\geq 1$ one shows that the second order differential
operator $I+\mathcal H_{\gamma}$ is globally hypoelliptic and its powers are
pseudo-differential operators with symbols
$$\sigma_{(I+\mathcal H_{\gamma})^{-\alpha/2}}(\ell)_{mn}=(1+m-\gamma m^2+\gamma\ell(\ell+1))^{-\alpha/2}\delta_{mn},$$
with $m,n\in\mathbb Z,\;
-\ell\leq m,n\leq\ell$.
We now have
$$
\|\sigma_{(I+\mathcal H_{\gamma})^{-\alpha/2}}(\ell)\|_{S_r}^r=
\Tr |\sigma_{(I+\mathcal H_{\gamma})^{-\alpha/2}}(\ell)|^r=
\sum_{m=-\ell}^\ell\left|1+m-\gamma m^2+\gamma\ell(\ell+1)\right|^{-\frac{\alpha r}{2}},
$$
where $\ell\in\N_0$. In order to estimate this sum, we consider the integral 
\beq
\int_{-R}^R |1+x-\gamma x^2+\gamma R^2+\gamma R|^{-\frac{\alpha r}{2}}dx.
\label{intschr}\eq
Using the inequality $ -R\leq x\leq R$ and from the identity 
$1+x-\gamma x^2+\gamma R^2+\gamma R =1+(R+x)+\gamma(R^2-x^2)+\gamma R-R,$
we get
$1+x-\gamma x^2+\gamma R^2+\gamma R\geq (\gamma -1)R$.
In particular, if $\gamma >1$ we obtain 
\[\int_{-R}^R(1+x-\gamma x^2+\gamma R^2+\gamma R)^{-\frac{\alpha r}{2}}dx\approx C R^{-\frac{\alpha r}{2}},\]
for large $R$.
Therefore, for $\gamma>1$, 
$$\sum_{m=-\ell}^\ell\left(1+m-\gamma m^2+\gamma\ell(\ell+1)\right)^{-\frac{\alpha r}{2}}$$ 
is of order $\ell^{-\frac{\alpha r}{2}}$.
Hence, $\|\sigma_{(I+\mathcal H_{\gamma})^{-\alpha/2}}(\ell)\|_{S_r}$
 is of order $\ell^{-\frac{\alpha}{2}}$ in this case.
 Consequently, if $\gamma > 1$ we obtain
 $$
 \sum_{[\xi]\in\widehat{\SO3}}d_{\xi} \|\sigma_{(I+\mathcal H_{\gamma})^{-\alpha/2}}(\xi)\|_{S_r}^r\approx
 C \sum_{\ell\in\N} \ell^{1-\frac{\alpha r}{2}}.
 $$
Thus, if $\gamma > 1$ we get 
$$(I+\mathcal H_{\gamma})^{-\alpha/2}\in S_r \textrm{ if and only if } \alpha r>4
\qquad (\gamma > 1).$$
Now, let us consider the case $0<\gamma \leq 1$.
If $-1\not\in\{\gamma\ell(\ell+1)-\gamma m^2+m): \ell\in\N_0, m\in\mathbb Z, |m|\leq\ell\},$ 
or, more generally, if 
 $-c\not\in\{\gamma\ell(\ell+1)-\gamma m^2+m): \ell\in\N_0, m\in\mathbb Z, |m|\leq\ell\},$ 
the operator $cI+\mathcal H_\gamma$ is invertible, and we can define its real powers as above.
Arguing as above, the corresponding modification of the integral (\ref{intschr}) does not decay with respect to $R$, 
so by the characterisation of Schatten classes in Theorem \ref{charnucSr} 
we get
$$(cI+\mathcal H_{\gamma})^{-\alpha/2}\notin S_r \;
\textrm{ for all } 0<r<\infty \textrm{ and } \alpha\in\mathbb R
\qquad (0<\gamma \leq 1).$$
A similar result holds then also on $\SU2\simeq\mathbb S^3$.

\section{A trace formula in the trace class $S_1(L^2(G))$ } 
\label{SEC:traces}
In this section we give trace formulae for operators on compact topological groups. We start by recalling the definition of the trace 
 of operators on Hilbert spaces.
 
Let $T:H\rightarrow H$ be an operator in $S_1(H)$ and let  $\{\phi_k\}_k$ be any orthonormal basis for the Hilbert space $H$. Then, the series $\sum\limits_{k=1}^{\infty}\langle T\phi_k,\phi_k\rangle_H$ is absolutely convergent and the sum is independent of the choice of the orthonormal basis $\{\phi_k\}_k$. Thus, we can define the trace $\Tr (T)$ of any linear operator
$T:H\rightarrow H$ in $S_1(H)$ by $$\Tr (T)=\sum\limits_{k=1}^{\infty}\langle T\phi_k,\phi_k\rangle_H,$$
where $\{\phi_k:k=1,2,\dots\}$ is any orthonormal basis for $H$. 
 
We will apply the definition above to the orthonormal basis of $L^2(G)$ given by
$$
\left\{ \sqrt{d_\xi}\,\xi_{ij}: \; 1\leq i,j\leq d_\xi,\; [\xi]\in\Gh \right\}.
$$
\begin{thm}\label{trace1} 
Let $G$ be a compact topological group. Let  $A$ be a left-invariant operator in $S_1(L^2(G))$ with matrix-valued symbol $\sigma_A(\xi)$.
Then its trace is given by 
\beq 
\label{trace}\Tr A=\sum\limits_{[\xi]\in\widehat{G} }d_{\xi}\Tr(\sigma_A(\xi)).
\eq
\end{thm}

\begin{proof} Let $A$ be a left-invariant operator which belongs to $S_1(L^2(G))$. We denote $\sigma=\sigma_A$, and the formula 
\[
K(x,y)=\sum\limits_{[\xi]\in\widehat{G} }d_{\xi} \Tr(\xi(x)\sigma(\xi)\xi(y)^* )
\]
represents the integral
kernel of $A$. We will calculate $\langle A\eta_{\ell m},\eta_{\ell m}\rangle _{L^2(G)}$ for $1\leq \ell,m\leq d_\eta,\; [\eta]\in\Gh$.
 We observe that
\[ \Tr(\xi(x)\sigma(\xi)\xi(y)^* )=\sum\limits_{i,j=1}^{d_{\xi}}(\xi(x)\sigma(\xi))_{ij}\overline{\xi(y)}_{ij}.\]
Hence
\begin{align*}
A\eta_{\ell m}(x)=&\int_G \sum\limits_{[\xi]\in\widehat{G} }d_{\xi} \Tr(\xi(x)\sigma(\xi)\xi(y)^* )\eta_{\ell m}(y)dy\\
=&\int_G \sum\limits_{[\xi]\in\widehat{G} }d_{\xi} \sum\limits_{i,j=1}^{d_{\xi}}(\xi(x)\sigma(\xi))_{ij}\overline{\xi(y)}_{ij}\eta_{\ell m}(y)dy\\
=&\int_G \sum\limits_{[\xi]\in\widehat{G} }d_{\xi} \sum\limits_{i,j=1}^{d_{\xi}}\sum\limits_{k=1}^{d_{\xi}}\xi(x)_{ik}\sigma(\xi)_{kj}\overline{\xi(y)}_{ij}\eta_{\ell m}(y)dy.\\
\end{align*}
Now, since $\langle \xi_{ik},\eta_{\ell m}\rangle _{L^2(G)}=d_\xi^{-1}\delta_{(i,j),(\ell, m)}$ by the orthonormality  of the system $\{\sqrt{d_\xi}\xi_{ij}\}$,  we obtain
\[\langle A\eta_{\ell m},\eta_{\ell m}\rangle _{L^2(G)}=d_{\eta}d_{\eta}^{-1}\sigma(\eta)_{mm}=\sigma(\eta)_{mm}\]
Therefore,
\begin{align*}
\sum\limits_{[\eta]\in\widehat{G} }\sum\limits_{\ell, m}\langle A\eta_{\ell m},\eta_{\ell m}\rangle _{L^2(G)}=\sum\limits_{[\eta]\in\widehat{G}}\sum\limits_{\ell, m}\sigma(\eta)_{mm}
=\sum\limits_{[\eta]\in\widehat{G} }d_{\eta}\Tr (\sigma(\eta)),
\end{align*}
concluding the proof.
\end{proof}


\end{document}